\documentclass[12pt]{article}

\usepackage{amssymb,amsthm}

\newcommand{\Z}{\mathbb Z}

\newtheorem*{thm}{Theorem}

\begin{document}

\title{Free by Cyclic groups are large}
\author{J.\,O.\,Button}
\maketitle
\begin{abstract}
If $F$ is a free group of finite rank at least 2 then any
group of the form $F$ by $\Z$ is large. In this short note
we show how the statement follows by combining a very recent
theorem of Hagen and Wise (using work of Agol and of Wise)
with earlier results of the author.
\end{abstract}

A finitely generated group is said to be large if it possesses
a finite index subgroup surjecting to a non abelian free group.
Groups of deficiency at least two were shown to be large in
\cite{bp} whereas $\Z$, $\Z\times\Z$ and the Klein bottle group
are groups of deficiency one which are not large.

Nevertheless there is both theoretical and experimental evidence
that most groups of deficiency one are large. An important case is
when we take a free group $F_n$ of rank $n\geq 2$ and an
automorphism $\alpha$ to form the free by cyclic group
$G=F_n\rtimes_\alpha\Z$ (for $n=0$ or 1 we obtain the three groups
above). In \cite{me} Theorem 5.1 we proved that if $G$ contains
$\Z\times\Z$ then $G$ is large. Now by \cite{bf}, \cite{bfc} and
\cite{bri}, if $G$ does not contain $\Z\times\Z$ then it is word
hyperbolic. The recent and extremely powerful work of Agol and
of Wise showing that a word hyperbolic group acting properly and
cocompactly on a CAT(0) cube complex is virtually special (has
a finite index subgroup embeddable in a right angled Artin group)
implies that the group (if non elementary) is large. This allowed
all fundamental groups of closed hyperbolic 3-manifolds to be proved
large.

Hence it is natural to try and show that if $G=F_n\rtimes_\alpha\Z$
is word hyperbolic for $n\geq 2$ then it acts properly and
cocompactly on a CAT(0) cube complex, which would then combine with
\cite{me} to prove largeness for all $G$. In \cite{hgw} this result
is obtained when $G$ is word hyperbolic and $\alpha$ is irreducible,
where $\alpha$ is reducible if we can write $F_n=A*B$ with $A$ and
$B$ proper free factors such that $\alpha^k(A)$ is a conjugate of
$A$ for some $k>0$. Consequently it would seem that we need to await
the extension of this result to word hyperbolic groups
$G=F_n\rtimes_\alpha\Z$ where $\alpha$ is reducible in order to
establish largeness for all free by cyclic groups. However another
theorem by the author allows this case to be covered.

\begin{thm}
If $G=F_n\rtimes_\alpha\Z$ for $F_n$ free of finite rank $n$ at least
2 then $G$ is large.
\end{thm}
\begin{proof}
We proceed by induction on $n$. The statement is known to hold for
$n=2$ by various means (for instance if $F_2$ is free on $a,b$ then
any automorphism sends $aba^{-1}b^{-1}$ to a conjugate of itself or
its inverse so $G$ contains $\Z\times\Z$). Suppose now any group
of the form $F_k\rtimes_\alpha\Z$ is known to be large when
$2\leq k\leq n-1$. If $G=F_n\rtimes_\alpha\Z$ is not word hyperbolic
or $\alpha$ is irreducible then we are done, so say $G=A*B$ for
$A,B$ proper free factors with $\alpha^k(A)$ conjugate to $\alpha$.
Without loss of generality we can replace $\alpha^k$ by $\alpha$
(dropping to a finite index subgroup of $G$) and then change
$\alpha$ by an inner automorphism (which does not now change
$G$) to get that $\alpha(A)=A$.

In \cite{me2} Theorem 2.1 it was shown that if all free by cyclic
groups have a finite index subgroup with first Betti number at
least 2 and $\alpha$ is reducible then $G$ is large. However, as
remarked after the end of the proof, a weaker condition is actually
used. Given an automorphism $\alpha$ of $F_n$ with $\alpha(A)=A$,
where $1\leq\mbox{rank}(A),\mbox{rank}(B)<n$, we can define the
reduced free by cyclic group $G_r=A\rtimes_\alpha\Z$ by restricting
$\alpha$ to $A$, and also the quotient free by cyclic group
$G_q=B\rtimes_{\overline{\alpha}}\Z$ by setting $\overline\alpha
=\pi\alpha$ where $\pi$ is the natural homomorphism from $A*B$ to $B$
(whereupon $\overline{\alpha}$ is an automorphism of $B$). Now if
both $G_r$ and $G_q$ have finite index subgroups with first Betti
number at least 2 then \cite{me2} Theorem 2.1 shows that
$G=(A*B)\rtimes_\alpha\Z$ is large. But using our inductive
hypothesis, either $G_r$ is of the form $F_k\rtimes_\alpha\Z$ for
$2\leq k<n$ and so is large, meaning that it has a finite index
subgroup with first Betti number at least 2, or 
$G_r\cong F_1\rtimes_\alpha\Z$ which, although this group is certainly
not large, contains $\Z\times\Z$ with index 1 or 2. The same is true
for $G_q$ so the induction is complete.
\end{proof}
  
This question was originally raised in the PhD of Martin Edjvet in
1984 as well as by the author (papers passim ad nauseam) and is
Question 9.30 (4) in \cite{afw}. It also answers positively
\cite{bpr} Question 12.16 (which is also Question 9.30 (3) in \cite{afw}
and is credited to Casson) which asks whether a group of the
form $F_n\rtimes_\alpha\Z$ for $n\geq 2$ has virtual first Betti
number at least 2. (Note that a stronger version of this question,
which was (F33) in \cite{ny} and also credited to Casson, was
answered negatively in \cite{sw}.)

We might also consider strictly ascending HNN extensions $(F_n)*_\theta$
where
$\theta:F_n\rightarrow F_n$ is injective but not surjective.
Although the result of Hagen and Wise also applies to word hyperbolic
groups where $\theta$ is irreducible, we have some way to go before
establishing largeness of all strictly ascending HNN extensions
$G=(F_n)*_\theta$ for $n\geq 2$. First we do not have an equivalent
result establishing word hyperbolicity in the absence of Baumslag
Solitar subgroups. Second, although \cite{me} Corollary 4.6 proved that if
$G=(F_n)*_\theta$ contains $\Z\times\Z$ then $G$ is large, if $\theta$
is not an automorphism then $\theta$ can contain Baumslag Solitar
subgroups of the form $BS(1,m)$ for $|m|>1$. Now \cite{me}
Theorem 4.7 establishes largeness of $G$ in this case, but only under
the further condition that the virtual first Betti number of $G$ is
at least 2. We note that a group containing $BS(1,m)$ for $|m|>1$
cannot be virtually special: indeed if $F_2$ is free on $a,b$  and
$\theta(a)=a^2,\theta(b)=b^2$ then $(F_2)*_\theta$ is not linear
over any field (\cite{wf}, \cite{ds}).

\end{document}